\documentclass[a4paper,12pt,draft]{amsart}
\usepackage[margin=30mm]{geometry}
\usepackage{amssymb, amsmath, amsthm, amscd}

\usepackage[T1]{fontenc}
\usepackage{eucal,mathrsfs,dsfont}
\usepackage{color}

\renewcommand{\le}{\leqslant}
\renewcommand{\ge}{\geqslant}

\definecolor{mno}{rgb}{0.5,0.1,0.5}

\newcommand{\R}{\mathds R}

\newcommand{\Ee}{\mathds E}
\newcommand{\ent}{\textup{Ent}}
\newcommand{\I}{\mathds 1}

\newtheorem{theorem}{Theorem}[section]

\newtheorem{proposition}[theorem]{Proposition}
\newtheorem{corollary}[theorem]{Corollary}

\theoremstyle{definition}

\newtheorem{example}[theorem]{Example}
\newtheorem{remark}[theorem]{Remark}

\begin{document}
\allowdisplaybreaks
\title[Poincar\'{e}-type Inequalities for Singular Stable-Like Dirichlet Forms] {\bfseries
Poincar\'{e}-type Inequalities for Singular Stable-Like
Dirichlet Forms}

\author{Jian Wang}

  \thanks{\emph{J.\ Wang:}
   School of Mathematics and Computer Science, Fujian Normal University, 350007 Fuzhou, P.R. China; Research Institute for Mathematical Sciences, Kyoto University, Kyoto 606-8502,
Japan. \texttt{jianwang@fjnu.edu.cn}}

\date{}

\maketitle

\begin{abstract} This paper is concerned with a class of singular stable-like Dirichlet forms on $\R^d$,
which are generated by $d$ independent copies of a one-dimensional
symmetric $\alpha$-stable process, and whose L\'{e}vy jump kernel
measure is concentrated on the union of the coordinate axes.
Explicit and sharp criteria for Poincar\'{e} inequality, super
Poincar\'{e} inequality and weak Poincar\'{e} inequality of such singular Dirichlet forms are
presented. When the reference measure is a product measure on
$\R^d$, we also consider the entropy inequality for the associated
Dirichlet forms, which is similar to the log-Sobolev inequality for
local Dirichlet forms, and enjoys the tensorisation property.
\medskip


\medskip

\end{abstract}
\allowdisplaybreaks
\section{Introduction and Main Results}\label{section1}
\subsection{Background for Functional Inequalities of Singular Stable-like Dirichlet Forms}\label{section1.1}
Let $\mu_V(dx)=e^{-V(x)}\,dx$ be a probability measure on $\R^d$, where $V\in C^1(\R^d)$. In the past few years functional inequalities for the following local Dirichlet form $(D_B,\mathscr{D}(D_B))$:
\begin{equation*}
\begin{split}
D_B(f,f)=&\,\int |\nabla f(x)|^2\,\mu_V(dx),\\
\mathscr{D}(D_B)=&\,\big\{f\in L^2(\R^d; \mu): D_B(f,f)<\infty\big\}
\end{split}
\end{equation*} have been intensely investigated by several probabilists.
One of the motivations comes from the study of the ergodicity for diffusion
semigroups associated with the second order elliptic operator
$$L_Bf=\Delta f-\nabla V\cdot \nabla f,$$ which is the generator of the Dirichlet form $(D_B,\mathscr{D}(D_B))$,
and also is the generator of the following stochastic differential
equation
$$dX_t=\sqrt{2}dB_t-\nabla V(X_t)\,dt,$$ where $(B_t)_{t\ge0}$ is a standard Brownian motion on $\R^d$.
Note that the coordinate processes of $(B_t)_{t\ge0}$ are $d$
independent copies of a one-dimensional Brownian motion.
One dimensional Brownian motion is a member of the class of
one dimensional strong Markov processes, called symmetric $\alpha$-stable
processes on $\R$. A one dimensional symmetric $\alpha$-stable process with
index $\alpha\in(0,2]$ is the L\'{e}vy process $(Y_t)_{t\ge0}$ so
that
$$\Ee\Big[e^{i\xi(Y_t-Y_0)}\Big]=e^{-t|\xi|^\alpha}\quad \textrm{ for }t>0\textrm{ and }\xi\in\R.$$
When $\alpha=2$, $(Y_t)_{t\ge0}$ is just a Brownian motion but running at twice
the speed. However
for $\alpha\in (0,2)$, $(Y_t)_{t\ge0}$ is a purely discontinuous
L\'{e}vy process with no drift, no Gaussian part, and L\'{e}vy
measure
$$n(dh)=c_\alpha/|h|^{1+\alpha}\,dh,$$ where $c_\alpha=\frac{\alpha 2^{\alpha-1}\Gamma((1+\alpha)/2)}{\pi^{1/2}\Gamma(1-\alpha/2)}.$
Recently there has been intense interest on the study of processes
with jumps. So it is natural to ask functional inequalities as above when
Brownian motion $(B_t)_{t\ge0}$ is replaced by $d$ independent
copies of a one-dimensional symmetric $\alpha$-stable process. This
is the topic of this paper.

We now give a more precise motivation of this paper. For any $t>0$,
let $Z_t=(Z_t^1, \cdots, Z_t^d)$ be a vector of $d$ independent
one-dimensional symmetric $\alpha$-stable processes with index
$\alpha\in(0,2)$. Consider the following stochastic differential
equation (SDE)
\begin{equation}\label{stable1}dX_t=dZ_t+b(X_{t})\,dt\end{equation} with
some regularity drift term $b$, such that the equation
\eqref{stable1} has a unique weak (or strong) solution
$(X_t)_{t\ge0}$ and a unique invariant probability measure $\mu$. (Surely there is a close relation between the drift term $b$ and the invariant measure $\mu$, and we will discuss it in another paper).
The existence of unique weak and strong solution to the SDE
\eqref{stable1} have been studied in \cite{Bass1} and \cite{Pr},
respectively. Now, let $(P_t)_{t\ge0}$ be the semigroup of the
process $(X_t)_{t\ge0}$ on $L^2(\R^d;\mu)$. Then, for any $f\in C_c^\infty(\R^d)$ and
$x\in \R^d$,
$$\lim_{t\to\infty} \frac{1}{t}\int_0^t P_sf (x)\,ds=\mu(f).$$ We
are interested in measuring the way of $P_t$ converging to its
equilibrium distribution $\mu$.

For simplicity, below we take the $L^2(\R^d; \mu)$-norm for example,
i.e. to consider the bound for $\|P_t(f)-\mu(f)\|_{L^2(\R^d;\mu)}$ for $f\in L^2(\R^d;\mu)$. Note
that the process $(Z_t)_{t\ge0}$ is a $d$-dimensional L\'{e}vy
process, and it picks a coordinate at random from $\{1,\ldots,d\}$
and then jumps a positive or negative distance in that direction.
Therefore, the L\'{e}vy measure for $(Z_t)_{t\ge0}$ is more singular
than that of the spherically symmetric $\alpha$-stable process (see $\nu_S(dz)$ below), and it is
concentrated on the union of the coordinate axes, a one-dimensional
subset of $\R^d$; that is, the density with respect to the Lebesgue measure of the L\'{e}vy measure for
the process $(Z_t)_{t\ge0}$ is given by
$$c_\alpha\left(\delta_{\{x_2=0, \ldots, x_d=0\}}\frac{1}{|x_1|^{1+\alpha}} +\ldots+\delta_{\{x_1=0, \ldots,
 x_{d-1}=0\}}\frac{1}{|x_d|^{1+\alpha}} \right),$$ where $\delta_A$ is Dirac measure of the set $A\subset\R^d$.
Hence, the generator of the process $(X_t)_{t\ge0}$ enjoys the
following expression
$$Lf(x)=\langle b(x), \nabla
f(x)\rangle+\sum_{i=1}^d\int_{\R}\Big(f(x+ze_i)-f(x)-z\I_{\{|z|\le
1\}}\nabla f(x)\cdot e_i
\Big)\frac{c_\alpha\,dz}{|z|^{1+\alpha}},$$where $e_i=(\overbrace{0,\ldots,0}^{i-1}, 1, \overbrace{0,\ldots,0}^{d-i})$ for $1\le
i\le d$. Furthermore, by some formal calculation, for every $f \in
C_c^{\infty}(\R^d)$ such that $\mu(f)=0$,
\begin{align*}
&\frac{d}{d t}\mu((P_t f)^2)\\
&=2\int (L P_t f
)P_t f \,d
\mu\\
&=-c_\alpha\sum_{i=1}^d\int_{\R^d\times \R}
\frac{\big(P_t f(x+ze_i)- P_t f(x)\big)^2}{|z|^{1+\alpha}}
\,d z\,\mu(d x)+\int L ((P_t f)^2)\,d\mu\\
&=-c_\alpha\sum_{i=1}^d\int_{\R^d\times \R}
\frac{\big(P_t f(x+ze_i)- P_t f(x)\big)^2}{|z|^{1+\alpha}}\,d
z\,\mu(d x), \end{align*} where in the second equality we have used
the fact that
$$L(f^2)=2f
Lf+{c_\alpha}\sum_{i=1}^d\int_{\R}\frac{\big(f(x+ze_i)-f(x)\big)^2}{|z|^{1+\alpha}}\,d
z,$$ and the third equality is due to that $\mu$ is an invariant
probability measure and $$\int L ((P_t  f)^2)\,d\mu=0.$$ Therefore, if
one can prove  \begin{equation}\label{stable-2}\mu(f^2)\le
\frac{Cc_\alpha}{2}\sum_{i=1}^d\int_{\R^d\times \R}
\frac{\big( f(x+ze_i)- f(x)\big)^2}{|z|^{1+\alpha}}\,d z\,\mu(d
x),\quad \mu(f)=0\end{equation} holds for some constant $C>0$, then for every $f \in
C_c^{\infty}(\R^d)$ such that $\mu(f)=0$,
$$\frac{d}{d t}\mu((P_t f)^2 )\le -\frac{2}{C}\mu((P_tf)^2).$$ This implies that for every $f \in
C_c^{\infty}(\R^d)$ such that $\mu(f)=0$, $$\mu((P_t f)^2)\le e^{-{2t}/{C}}\mu(f^2),\quad t>0.$$ In particular, we have
$$\|P_t f-\mu(f)\|_{L^2(\R^d;\mu)}\le e^{-t/C}\|f-\mu(f)\|_{L^2(\R^d;\mu)},\quad t>0, f\in L^2(\R^d; \mu),$$which is our
desired assertion. In particular, \eqref{stable-2}
motivates us to study the following bilinear form
\begin{equation*}
\begin{split}
D(f,f)=&\,\frac{1}{2}\sum_{i=1}^d\int_{\R^d\times \R}
\frac{(f(x+ze_i)-f(x))^2}{|z|^{1+\alpha}}\,dz\,\mu(dx).
\end{split}
\end{equation*} Let $\mathscr{D}(D)$ be the closure of $C_c^\infty(\R^d)$
under the $D_1$-norm
$$\|f\|_{D_1}:=({\|f\|_{L^2(\R^d;\mu)}^2+D(f,f)})^{1/2}.$$ Then,
according to \cite[Example 1.2.4]{FOT}, $(D,\mathscr{D}(D))$ is regular symmetric Dirichlet form on $L^2(\R^d;\mu)$. In
this setting, \eqref{stable-2} is the Poincar\'{e} inequality
for the Dirichlet form $(D,\mathscr{D}(D))$, which will be stated explicitly in Theorem \ref{thm1} below.

The main goal of this paper is to prove various Poincar\'{e} type
inequalities for the Dirichlet form $(D,\mathscr{D}(D))$. Recently,
explicit criteria have been presented in \cite{CW14, CWW14, WW} for functional
inequalities of the following (standard) stable-like Dirichlet form
\begin{equation*}
\begin{split}
D_S(f,f)=&\,\frac{1}{2}\int_{\R^d\times \R^d}
\frac{(f(x+z)-f(x))^2}{|z|^{d+\alpha}}\,dz\,\mu(dx), \quad f\in
\mathscr{D}(D_S),
\end{split}
\end{equation*} where $\mathscr{D}(D_S)$ is the closure of $C_c^\infty(\R^d)$
under the $D_{S,1}$-norm
$$\|f\|_{D_{S,1}}:=({\|f\|_{L^2(\R^d;\mu)}^2+D_S(f,f)})^{1/2}.$$
Different from the singular stable-like Dirichlet form
$(D,\mathscr{D}(D))$ considered in the present paper, the L\'{e}vy
jump kernel of $(D_S, \mathscr{D}(D_S))$ is associated with
spherically symmetric $\alpha$-stable processes in \cite{WW}, and it is given
by
$$\nu_S(dz)=\frac{C_{d,\alpha}}{|z|^{d+\alpha}}\,dz,$$ where $C_{d,\alpha}=\frac{\alpha 2^{\alpha-1} \Gamma((d+\alpha)/2)}{\pi^{d/2}\Gamma(1-\alpha/2)}$.

Comparing with the method for the proofs of Poincar\'{e} type
inequalities for $(D_{S},\mathscr{D}(D_S))$ in \cite{WW}, in order to
get the corresponding functional inequalities for
$(D,\mathscr{D}(D))$, we will face with two fundamental differences:
\begin{itemize}
\item [(1)] The efficient approach to yield functional inequalities for $(D_{S},\mathscr{D}(D_S))$ is
to check the Lyapunov type condition for the associated generator,
which heavily depends on the corresponding L\'{e}vy jump kernel
$\nu_S(dz)$. In particular, the Lyapunov function $\phi$ we choose
in \cite{WW} is of the form $\phi(x)=|x|^\beta$ with some constant
$\beta\in(0,1\wedge\alpha)$ for $|x|$ large enough.  Such test function $\phi$ is useful for the generator of
$(D,\mathscr{D}(D))$, but the argument of \cite[Proposition 2.3]{WW} does not work in the present setting.

\item[(2)] Another ingredient to obtain Poincar\'{e} inequality
and super Poincar\'{e} inequality for $(D_{S},\mathscr{D}(D_S))$ is
to prove the local Poincar\'{e} inequality and the local super
Poincar\'{e} inequality. The local super Poincar\'{e} inequality for
$(D_{S},\mathscr{D}(D_S))$ is derived from the classical Sobolev
inequality for fractional Laplacians; while the local Poincar\'{e}
inequality for $(D_{S},\mathscr{D}(D_S))$ is easily obtained by
applying the Cauchy-Schwarz inequality to local variance. However we are unable to use
these approaches here, since the L\'{e}vy jump kernel for the
Dirichlet form $(D,\mathscr{D}(D))$ is much more singular.
\end{itemize}

Due to the above differences and difficulties, obtaining the
criteria for Poincar\'{e} inequality and super Poincar\'{e}
inequality for $(D,\mathscr{D}(D))$ requires new approaches and
ideas, which include the following two points:
\begin{itemize}
\item [(3)] The new choice of  Lyapunov function $\phi$ (a little different from that in \cite{WW}) for the generator associated with $(D,\mathscr{D}(D))$.
 Some more refined calculations are required, due to the character of L\'{e}vy jump kernel for the
Dirichlet form $(D,\mathscr{D}(D))$. (See Proposition \ref{pro2} and its proof.)

\item[(4)] The local super
Poincar\'{e} inequality for $(D,\mathscr{D}(D))$, where
    a new direct proof of the local functional inequality for singular non-local Dirichlet forms is given.
    (See Proposition \ref{pro3}.) An application of the more recent result on the equivalence of defective Poincar\'{e}
    inequality and true Poincar\'{e} inequality for symmetric conservative Dirichlet forms developed in \cite{W13}, yields the
    desired Poincar\'{e} inequality for $(D,\mathscr{D}(D))$, and circumvents the difficulty of
    proving the local Poincar\'{e} inequality for
    $(D,\mathscr{D}(D))$. (See the proof of Theorem \ref{thm1}(1).)
\end{itemize}

\subsection{Main Results: Criteria for Functional Inequalities of Singular Stable-like Dirichlet Forms}
Throughout the paper, we always suppose that
\emph{$\mu(dx)=\mu_V(dx):=e^{-V}\,dx$ is a probability measure on
$\R^d$, such that $e^{-V}$ is a bounded measurable function on
$\R^d.$} We emphasize that unlike \cite{WW} no $C^1$-regularity on $V$ is needed in the present paper. For any $x\in\R^d$, set $$\Gamma^V_{\inf}(x):= \inf_{1\le i\le d:\, |x_i|\ge |x|/\sqrt{d}}\,\,\inf_{|u_i|\le 1} e^{-V(x_1,\cdots, x_{i-1}, u_i, x_{i+1},\cdots,x_d)},$$ $$\Gamma^V_{\sup}(x):= \sup_{1\le i\le d}\,\,\sup_{|u_i|\ge |x_i|} e^{-V(x_1,\cdots, x_{i-1}, u_i, x_{i+1},\cdots,x_d)}$$ and
$$\Lambda(x):=\frac{e^{V(x)}\Gamma^V_{\inf}(x)}{(1+|x|)^{1+\alpha}}.$$
Furthermore, we define $$\Phi(r)=\inf_{|x|\ge r} \Lambda(x),\quad r>0.$$

We are now in a position to state the main result in our paper.
\begin{theorem}\label{thm1} Suppose that there exists a constant $\gamma\in (0, \alpha\wedge1)$ such that
\begin{equation}\label{thm1-0}\limsup_{|x|\to\infty}
\frac{|x|^{1+\alpha-\gamma} \Gamma^V_{\sup}(x)}{\Gamma^V_{\inf}(x)}=0.\end{equation} We have the
following statements.
\begin{itemize}
\item[(1)] If $\lim_{r\to\infty}\Phi(r)>0$, then the following Poincar\'{e} inequality
\begin{equation}\label{thm1-1}\mu_V(f^2)\le CD(f,f)+\mu_V(f)^2 ,\quad f\in \mathscr{D}(D)\end{equation} holds with some constant $C>0$.
\item[(2)] If $\lim_{r\to\infty}\Phi(r)=\infty$, then the following super Poincar\'{e} inequality
\begin{equation}\label{thm1-2}\mu_V(f^2)\le sD(f,f)+\beta(s)\mu_V(|f|)^2,\quad s>0, f\in \mathscr{D}(D)\end{equation} holds with $$\beta(s)=C_1(1+s^{-d/\alpha})\frac{\big(\sup_{|x|\le 2\sqrt{d}\Phi^{-1}(C_2(1+1/s))}e^{V(x)}\big)^{2+d/\alpha}}{\big(\inf_{|x|\le \Phi^{-1}(C_2(1+1/s))} e^{V(x)}\big)^{1+d/\alpha}}$$ for some constants $C_1$ and $C_2>0$, where $\Phi^{-1}$ is the generalized inverse of $\Phi$, i.e.\ $\Phi^{-1}(r)=\inf\{s\ge0: \Phi(s)\ge r\}.$
\end{itemize}\end{theorem}

The following corollary shows that Theorem \ref{thm1} is sharp in some situation.

\begin{corollary}\label{exm1} Let $$e^{-V(x)}=C_{\varepsilon_1, \ldots, \varepsilon_d}\prod_{i=1}^d(1+|x_i|)^{-(1+\varepsilon_i)}$$ with $\varepsilon_i>0$ for all $1\le i\le d$.
\begin{itemize}
\item[(1)] The Poincar\'{e} inequality \eqref{thm1-1} holds with some
constant $C>0$ if and only if $\varepsilon_i\ge \alpha$ for all $1\le i\le d$.
\item[(2)] The super Poincar\'{e} inequality \eqref{thm1-2} holds with
some function $\beta:(0,\infty)\to(0,\infty)$ if and only if
$\varepsilon_i>\alpha$ for all $1\le i\le d$, and in this case there exists a constant $c>0$
such that the super Poincar\'{e} inequality \eqref{thm1-2} holds with
$$\beta(r)\le c \left(1+r^{-\big(\frac{d}{\alpha}+\frac{(2\alpha+d)\sum_{i=1}^d(1+\varepsilon_i)}{\alpha(\varepsilon_*-\alpha)}\big)}\right),\quad
r>0,$$ where $\varepsilon_*=\min_{1\le i\le d} \varepsilon_i$; and
equivalently,
\begin{equation*}
\begin{split}\|P_t\|_{L^1(\R^d;\mu_V)\to L^\infty(\R^d;\mu_V)}:&=\sup_{f\in L^1(\R^d;\mu_V)}\|P_tf\|_{L^\infty(\R^d;\mu_V)} \\
&\le
\lambda\left(1+t^{-\big(\frac{d}{\alpha}+\frac{(2\alpha+d)\sum_{i=1}^d(1+\varepsilon_i)}{\alpha(\varepsilon_*-\alpha)}\big)}\right),\quad
t>0\end{split}
\end{equation*} holds with some constant $\lambda>0$.
\end{itemize}

\end{corollary}
We give two remarks on Corollary \ref{exm1}, which point out the difference between Dirichlet form
$(D,\mathscr{D}(D))$ and $(D_S,\mathscr{D}(D_S))$. Since $(D,\mathscr{D}(D))$ is just $(D_S,\mathscr{D}(D_S))$ when $d=1$, we assume that $d\ge 2$ below. Both reference measures $\mu$ in $(D,\mathscr{D}(D))$ and $(D_S,\mathscr{D}(D_S))$ are given by $\mu(dx)=\mu_V(dx)=e^{-V(x)}\,dx.$
\begin{itemize}
\item[(i)] Let $$e^{-V(x)}=C_{\varepsilon_1, \ldots, \varepsilon_d}\prod_{i=1}^d(1+|x_i|)^{-(1+\varepsilon_i)}$$ with $\varepsilon_i\ge \alpha$ for all $1\le i\le d$. We know from Corollary \ref{exm1} above that the Poincar\'{e} inequality holds for $(D,\mathscr{D}(D));$ however, we do not know whether
the Poincar\'{e} inequality holds for the standard Dirichlet form
$(D_S,\mathscr{D}(D_S))$, because the assumptions of \cite[Theorem 1.1]{WW} are not satisfied.
\item[(ii)]  Let
$$e^{-V(x)}=C_{\varepsilon,d}(1+|x|)^{-(d+\varepsilon)}$$ with
$\varepsilon>0$. Then, according to \cite[Corollary 1.2]{WW}, we
know that the Poincar\'{e} inequality holds for the Dirichlet form
$(D_S,\mathscr{D}(D_S))$ if and only if $\varepsilon\ge\alpha$, and
the super Poincar\'{e} inequality holds for $(D_S,\mathscr{D}(D_S))$
if and only if $\varepsilon>\alpha$. On the other hand, we do not know whether
the Poincar\'{e} inequality holds for the Dirichlet form
$(D,\mathscr{D}(D))$ for any $\varepsilon>0$, since the assumptions of Theorem
\ref{thm1} do not hold for all $\varepsilon>0$.
\end{itemize}

The following corollary further indicates that the conclusions of Theorem \ref{thm1} are explicit in some setting.

\begin{corollary}\label{exm2} Let $$e^{-V(x)}=C_{\varepsilon_1, \ldots, \varepsilon_d,\alpha}\prod_{i=1}^d(1+|x_i|)^{-(1+\alpha)}\log^{-\varepsilon_i}(e+|x_i|)$$ with $\varepsilon_i\in\R$ for all $1\le i\le d$.
\begin{itemize}
\item[(1)] The Poincar\'{e} inequality \eqref{thm1-1} holds for some
constant $C>0$ if and only if $\varepsilon_i\ge 0$ for all $1\le i\le d$.
\item[(2)] The super Poincar\'{e} inequality \eqref{thm1-2} holds for
some function $\beta:(0,\infty)\to(0,\infty)$ if and only if
$\varepsilon_i>0$ for all $1\le i\le d$, and in this case there exists a constant $c>0$
such that the super Poincar\'{e} inequality \eqref{thm1-2} holds with
$$\beta(r)\le \exp\Big(c\big(1+r^{-1/\varepsilon_*}\big)\Big),\quad
r>0$$ for some constants $c>0$ and $\varepsilon_*=\min_{1\le i\le d}
\varepsilon_i$, so that when $\varepsilon_*>1$,
$$\|P_t\|_{L^1(\R^d;\mu_V)\to L^\infty(\R^d;\mu_V)}\le
\exp\Big(\lambda(1+t^{-1/(\varepsilon_*-1)})\Big),\quad t>0$$ holds
for some constant $\lambda>0$. The rate function $\beta$ above is
sharp in the sense that \eqref{thm1-2} does not hold if
$$\lim_{r\to0} r^{1/\varepsilon_*}\log \beta(r)=0.$$

In particular, the following log-Sobolev inequality
$$\mu_V(f^2\log f^2)\le CD(f,f),\quad f\in \mathscr{D}(D), \mu_V(f^2)=1$$
holds for some constant $C>0$ if and only if $\varepsilon_i\ge1$ for all $1\le i\le d$.
\end{itemize}
\end{corollary}

Both Corollaries \ref{exm1} and \ref{exm2} are concerned with Poincar\'{e}-type inequalities for product measures. As mentioned in the remark after the proof of Theorem \ref{thm10} in Section \ref{section3}, Poincar\'{e} inequalities for Corollaries \ref{exm1}(1) and \ref{exm2}(1) can be obtained from the results of \cite{WW} in one-dimensional setting and the well-known tensorisation procedure. In the following example, we consider product measure with variable order. We mention that the Poincar\'{e} inequality for such measure can not be deduced directly by tensorisation argument.

 \begin{example}\label{coro} Let $$e^{-V(x)}=C\prod_{i=1}^d(1+|x_i|)^{-(1+a_i(x))},$$ where for all $1\le i\le d$, $a_i$ is a bounded Borel measurable function such that $\inf_{x\in \R^d} a_i(x)>0$, and $C>0$ is the normalizing constant. Define
 $$a_j^*(x|u_i):=a_j((x_1, \ldots, x_{i-1}, u_i, x_{i+1},\ldots x_d)),\quad x\in\R^d, u_i\in \R, 1\le i, j\le d.$$ Suppose that for $|x|$ large enough and for all $1\le j\le d$,
 \begin{equation}\label{coro-22}M_j(x):=\inf_{1\le i\le d}\inf_{|u_i|\ge |x_i|}a_j^*(x|u_i)\ge \sup_{1\le i\le d:\, |x_i|\ge |x|/\sqrt{d} } \sup_{|u_i|\le 1} a_j^*(x|u_i)=:N_j(x).\end{equation} Then, we have the following statements.
 \begin{itemize}
 \item[(1)]  If for $|x|$ large enough, \begin{equation}\label{coro-1}  A(x):=\inf_{1\le i\le d:|x_i|\ge|x|/\sqrt{d}}\sup_{|u_i|\le 1}a_i^*(x|u_i)\ge \alpha,\end{equation} then the Poincar\'{e} inequality \eqref{thm1-1} holds with some constant $C>0$.

\item[(2)]
 If \begin{equation*}\label{coro-3}A^*:=\liminf_{|x|\to\infty}A(x)> \alpha,\end{equation*} then for any $\varepsilon>0$, there is a constant $c=c(\varepsilon)>0$ such that the super Poincar\'{e} inequality \eqref{thm1-2} holds with
\begin{equation}\label{beta}\beta(r)\le c \left(1+r^{-\big(\frac{d}{\alpha}+\frac{(2\alpha+d)\sum_{i=1}^d(1+B_i)}{\alpha(A^{*}-\varepsilon- \alpha)}\big)}\right),\quad
r>0,\end{equation} where $B_i= \sup_{x\in \R^d} a_i(x)$ for all $1\le i\le d.$ If moreover $\inf_{|x|\ge r}A(x)$ is constant for $r$ large enough, then \eqref{beta} is satisfied with $\varepsilon=0,$ i.e.\ $$\beta(r)\le c_0 \left(1+r^{-\big(\frac{d}{\alpha}+\frac{(2\alpha+d)\sum_{i=1}^d(1+B_i)}{\alpha(A^{*}- \alpha)}\big)}\right),\quad
r>0.$$
\end{itemize}
\end{example}

\medskip

The remainder of this paper is arranged as follows. The next section
is devoted to the preliminary analysis on singular stable-like Dirichlet form
$(D,\mathscr{D}(D))$. The Lyapunov type drift condition for the
associated truncated Dirichlet form $(D_{>1},\mathscr{D}(D_{>1}))$
is established, and the local super Poincar\'{e} inequality for
$(D,\mathscr{D}(D))$ is also presented. In Section \ref{section3},
we will prove Theorem \ref{thm1} and Corollaries \ref{exm1}, \ref{exm2} and Example \ref{coro}. In particular, on the one hand, when the reference measure is a product measure on
$\R^d$, the entropy inequality for $(D,\mathscr{D}(D))$ is considered here, which shows that $(D,\mathscr{D}(D))$ enjoys the tensorisation property; on the other hand, the weak
Poincar\'{e} inequality for $(D,\mathscr{D}(D))$ is also included,
which can be regarded as a complement of Theorem \ref{thm1}.

\section{Preliminary  Analysis on Singular Stable-like Dirichlet Forms}
Let $\mu_V(dx)=e^{-V(x)}\,dx$ be a probability measure on $\R^d$ such that $e^{-V}$ is a bounded measurable function. For any $f\in C_b^1(\R^d)$, since
\begin{equation*}
\begin{split}
\sum_{i=1}^d\int_{\R^d\times \R}& \frac{(f(x+ze_i)-f(x))^2}{|z|^{1+\alpha}}\,dz\,\mu_V(dx)\\
&\le 4(\|f\|_\infty\vee\|\nabla f\|_\infty)^2\sum_{i=1}^d\int_{\R^d\times \R} \frac{1\wedge z^2}{|z|^{1+\alpha}}\,dz\,\mu_V(dx)\\
&<\infty,
\end{split}
\end{equation*}
we can well define
$$
D(f,f):=\,\frac{1}{2}\sum_{i=1}^d\int_{\R^d\times \R}
\frac{(f(x+ze_i)-f(x))^2}{|z|^{1+\alpha}}\,dz\,\mu_V(dx).$$

For any $x$, $y\in\R^d$, set
\begin{equation*}
\begin{split} &J(x,y):=\frac{1}{2}\big(e^{V(x)}+e^{V(y)}\big)\\
&\times\left(\delta_{\{x_2-y_2=0, \ldots,
x_d-y_d=0\}}\frac{1}{|x_1-y_1|^{1+\alpha}} +\ldots+\delta_{\{x_1-y_1=0,
\ldots,
 x_{d-1}-y_{d-1}=0\}}\frac{1}{|x_d-y_d|^{1+\alpha}} \right),\end{split}
\end{equation*}where $\delta_A$ is Dirac measure of the set $A$.
Then, $$D(f,f)=\frac{1}{2}\int_{\R^d\times \R^d}
(f(x)-f(y))^2J(x,y)\,\mu_V(dx)\,\mu_V(dy),$$ and $J(x,y)$ is the associated L\'{e}vy jump kernel measure. Furthermore, it is easy
to check that
$$x\mapsto \int\big(1\wedge |x-y|^2)J(x,y)\,\mu_V(dy)\in L^1(\R^d; \mu_V).$$ Let
$\mathscr{D}(D)$ be the closure of $C_c^\infty(\R^d)$
under the $D_1$-norm
$$\|f\|_{D_1}:=\sqrt{{\|f\|_{L^2(\R^d;\mu_V)}^2+D(f,f)}}.$$ Then, we
know from \cite[Example 1.2.4]{FOT} that $(D,\mathscr{D}(D))$ is a
regular symmetric Dirichlet form on $L^2(\R^d;\mu_V)$. Furthermore, denote by $(P_t)_{t\ge0}$ the semigroup on $L^2(\R^d;\mu_V)$ associated with $(D,\mathscr{D}(D))$, which can be extended into $L^\infty(\R^d;\mu_V)$, e.g.\ see \cite[Page 56]{FOT}. Since $\mu_V$ is a symmetric and invariant probability measure of $(P_t)_{t\ge0}$, $1=\mu_V (1)=\mu_V(P_t 1)$ for each $t>0$, which implies that $P_t1(x)=1$ for all $t>0$ and almost all $x\in \R^d$. Then, the Dirichlet form $(D,\mathscr{D}(D))$ is conservative.
\medskip

To deal with functional inequalities for the Dirichlet form
$(D,\mathscr{D}(D))$, we will make full use of the truncation
approach. For this, we define for any $f\in C_b^1(\R^d)$,
\begin{equation*}
\begin{split}
D_{>1}(f,f):=&\,\frac{1}{2}\sum_{i=1}^d\int_{\{\R^d\times \R:
|z|>1\}} \frac{(f(x+ze_i)-f(x))^2}{|z|^{1+\alpha}}\,dz\,\mu_V(dx).
\end{split}
\end{equation*} Let $\mathscr{D}(D_{>1})$ be the closure of $C_c^\infty(\R^d)$
under the norm $$\|f\|_{D_{>1},1}:=\sqrt{{\|f\|_{L^2(\R^d;\mu_V)}^2+D_{>1}(f,f)}}.$$ It
is clear that $D_{>1}(f,f)\le D(f,f)$, and so
$\mathscr{D}(D)\subset\mathscr{D}(D_{>1})$. From this, we also can easily conclude that
$(D_{>1},\mathscr{D}(D_{>1}))$ is a regular symmetric
Dirichlet form on $L^2(\R^d;\mu_V)$.

Denote by $B(\R^d)$ the set of measurable functions on $\R^d$, and by $B_b(\R^d)$ the set of bounded measurable functions on $\R^d$. For any $f\in B_b(\R^d)$, define
$$L_{>1}f(x):=\frac{1}{2}\sum_{i=1}^d\int_{\{|z|> 1\}}\big(f(x+ze_i)-f(x)\big)\frac{e^{V(x)-V(x+ze_i)}+1}{|z|^{1+\alpha}}\,dz.$$
We have
\begin{proposition}\label{pro1} $(1)$ For any $f$, $g\in B_b(\R^d)$,
$$D_{>1}(f,g)=-\int fL_{>1}g\,d\mu_V.$$

$(2)$ For $0<\gamma<\alpha$, let \begin{equation*}
\begin{split}\mathscr{C_\gamma}:=\Big\{&g\in
B(\R^d): \textrm{ there exists a constant } C>0 \textrm{ such that
}\\&|g(x)-g(y)|\le C|x-y|^\gamma \textrm{ for any }x, y\in \R^d
\textrm{ with } |x-y|>1\Big\}.\end{split}
\end{equation*} Then, $B_b(\R^d)\subset\mathscr{C_\gamma}$, and for any $g\in
\mathscr{C_\gamma}$, $L_{>1}g$ exists pointwise as a locally bounded
function. Moreover for any $f\in B_b(\R^d)$ and any $g\in
\mathscr{C_\gamma}$,
$$\sum_{i=1}^d\int_{\{\R^d\times\R: |z|>1\}}\frac{|f(x+ze_i)-f(x)||g(x+ze_i)-g(x)|}{|z|^{1+\alpha}}\,dz\,\mu_V(dx)<\infty$$ and
\begin{equation*}
\begin{split}-\int f(x)&L_{>1}g(x)\,\mu_V(dx)\\
&= \frac{1}{2} \sum_{i=1}^d\int_{\{\R^d\times\R: |z|>1\}}\frac{(f(x+ze_i)-f(x))(g(x+ze_i)-g(x))}{|z|^{1+\alpha}}\,dz\,\mu_V(dx).\end{split}
\end{equation*} \end{proposition}

\begin{proof}(1) We first note that for any $g\in B_b(\R^d)$ and $x\in \R^d$,
\begin{equation*}
\begin{split}|L_{>1}g(x)|&\le \frac{1+\|e^{-V}\|_\infty e^{V(x)}}{2}\sum_{i=1}^d\int_{\{|z|>1\}}\frac{|g(x+ze_i)-g(x)|}{|z|^{1+\alpha}}\,dz\\
&\le \frac{2d}{\alpha}\big(1+\|e^{-V}\|_\infty e^{V(x)}\big)\|g\|_\infty,\end{split}
\end{equation*}
which implies that $L_{>1}g$ is well defined.
On the other hand, for any $f$, $g\in B_b(\R^d)$,
\begin{align*}\Big|\int fL_{>1} g\,d\mu_V\Big|\le &\|f\|_\infty\|g\|_\infty \sum_{i=1}^d
\int_{\{\R^d\times \R: |z|>1\}}\frac{1}{|z|^{1+\alpha}}(e^{-V(x)}+e^{-V(x+ze_i)})\,dz\,dx\\
=& \|f\|_\infty\|g\|_\infty \sum_{i=1}^d
\int_{\{\R^d\times \R: |z|>1\}}\frac{1}{|z|^{1+\alpha}}e^{-V(x)}\,dz\,dx\\
&+ \|f\|_\infty\|g\|_\infty \sum_{i=1}^d
\int_{\{\R^d\times \R: |z|>1\}}\frac{1}{|z|^{1+\alpha}}e^{-V(x+ze_i)}\,dz\,dx\\
=&2\|f\|_\infty\|g\|_\infty \sum_{i=1}^d\int_{\{\R^d\times \R: |z|>1\}}\frac{e^{-V(x)}}{|z|^{1+\alpha}}\,dz\,dx\\
=&\frac{4d\|f\|_\infty\|g\|_\infty}{\alpha},\end{align*}
where the second equality follows from the change of the variable
$x\mapsto x+ze_i$ in the second term of the first
equality. Thus, $\int fL_{>1} g\,d\mu_V$ is also well defined.

Furthermore, for any $f$, $g\in B_b(\R^d)$,
\begin{equation*}
\begin{split}-\int &f(x)L_{>1}g(x)\,\mu_V(dx)\\
&= -\frac{1}{2}\sum_{i=1}^d\int_{\{\R^d\times\R:
|z|>1\}}\frac{f(x)(g(x+ze_i)-g(x))}{|z|^{1+\alpha}}\Big(e^{-V(x)}+e^{-V(x+ze_i)}\Big)\,dz\,d
x.\end{split}
\end{equation*}Changing the variables $z\to -z$ and $x\to x+ze_i$ in the right hand side, we get
\begin{equation*}
\begin{split}& -\int f(x)L_{>1}g(x)\,\mu_V(dx)\\
&= -\frac{1}{2}\sum_{i=1}^d\int_{\{\R^d\times\R:
|z|>1\}}\frac{f(x+ze_i)(g(x)-g(x+ze_i))}{|z|^{1+\alpha}}\Big(e^{-V(x)}+e^{-V(x+ze_i)}\Big)\,dz\,d
x.\end{split}
\end{equation*}
Therefore, for any $f$, $g\in B_b(\R^d)$,
\begin{align*} &-\int f(x)L_{>1}g(x)\,\mu_V(dx)\\
&= \frac{1}{4}\sum_{i=1}^d\int_{\{\R^d\times\R:
|z|>1\}}\frac{(f(x+ze_i)-f(x))(g(x+ze_i)-g(x))}{|z|^{1+\alpha}}\\
&\qquad\qquad\qquad\qquad\qquad\qquad \times\Big(e^{-V(x)}+e^{-V(x+ze_i)}\Big)\,dz\,d
x\\
& =\frac{1}{4}\sum_{i=1}^d\int_{\{\R^d\times\R:
|z|>1\}}\frac{(f(x+ze_i)-f(x))(g(x+ze_i)-g(x))}{|z|^{1+\alpha}}e^{-V(x)}\,dz\,d
x \\
&\quad+ \frac{1}{4}\sum_{i=1}^d\int_{\{\R^d\times\R:
|z|>1\}}\frac{(f(x+ze_i)-f(x))(g(x+ze_i)-g(x))}{|z|^{1+\alpha}}e^{-V(x+ze_i)}\,dz\,d
x\\
&=\frac{1}{2}\sum_{i=1}^d\int_{\{\R^d\times\R:
|z|>1\}}\frac{(f(x+ze_i)-f(x))(g(x+ze_i)-g(x))}{|z|^{1+\alpha}}\,dz\,\mu_V(d
x)\\
&=D_{>1}(f,g),\end{align*}where the third equality follows from the change of variables $z\to -z$ and $x\to
x+ze_i$ again in the second term of the second equality.

(2) By the definition of $\mathscr{C_\gamma}$, it is easy to see that $B_b(\R^d)\subset\mathscr{C_\gamma}$. For any $g\in \mathscr{C}_\gamma$,
\begin{equation*}
\begin{split}|L_{>1}g(x)|\le
&\frac{1}{2}\big(1+\|e^{-V}\|_\infty
e^{V(x)}\big)\sum_{i=1}^d\int_{\{|z|>
1\}}\big|g(x+ze_i)-g(x)\big|\frac{1}{|z|^{1+\alpha}}\,dz\\
\le&\frac{dC}{\alpha-\gamma}\big(1+\|e^{-V}\|_\infty
e^{V(x)}\big),\end{split}
\end{equation*} from which we know that $L_{>1}g(x)$ is
a well-defined and locally bounded function on $\R^d$. On the other
hand,
\begin{equation*}
\begin{split} &\sum_{i=1}^d\int_{\{\R^d\times\R: |z|>1\}}\frac{|f(x+ze_i)-f(x)||g(x+ze_i)-g(x)|}
{|z|^{1+\alpha}}\,dz\,\mu_V(dx)\\
&\le C\sum_{i=1}^d\int_{\{\R^d\times\R: |z|>1\}}
\frac{|f(x+ze_i)-f(x)|}
{|z|^{1+\alpha-\gamma}}\,dz\,\mu_V(dx)\\
&\le \frac{4Cd\|f\|_\infty}{\alpha-\gamma}.\end{split}
\end{equation*} Then the last assertion follows from the argument in part (1).
\end{proof}

\begin{remark}According to Proposition \ref{pro1}(1), the operator $L_{>1}$ is a
(formal) generator of the Dirichlet form
$(D_{>1},\mathscr{D}(D_>1))$. However, to verify that the operator
$L_{>1}$ maps $B_b(\R^d)$ into
$L^2(\R^d;\mu_V)$, we need some additional assumption. For sufficient conditions that $L_{>1}$ maps $C_c^\infty(\R^d)$ into
$L^2(\R^d;\mu_V)$, one can refer to
\cite[Theorem 2.1]{CW13}.
\end{remark}

Let $\gamma\in(0,1\wedge\alpha)$ and define $$\phi(x)=1+\sum_{i=1}^d|x_i|^\gamma,$$ which is the Lyapunov function we mentioned in Remark (3) in the end of Section \ref{section1.1}. Then, for any $x,y\in\R^d$, $$|\phi(x)-\phi(y)|\le \sum_{i=1}^d\big||x_i|^\gamma-|y_i|^\gamma\big|\le \sum_{i=1}^d|x_i-y_i|^\gamma\le d |x-y|^\gamma,$$ and so $\phi\in \mathscr{C}_\gamma$.  It follows from Proposition \ref{pro1}(2) that $L_{>1}\phi(x)$ is a well-defined locally bounded function on $\R^d$. Indeed, we have the following explicit estimate for $L_{>1}\phi(x)$.

\begin{proposition}\label{pro2}
For any $x\in\R^d$, let $\Gamma^V_{\inf}(x)$, $\Gamma^V_{\sup}(x)$ and $\Lambda(x)$ be these defined in Theorem $\ref{thm1}$.
Suppose that \eqref{thm1-0} holds with some $\gamma\in (0,\alpha\wedge1)$, and \begin{equation}\label{pro2-1}
\liminf_{|x|\to\infty}  \frac{e^{V(x)}\Gamma^V_{\inf}(x)}{|x|^{1+\alpha}}>0.
\end{equation}
Then, there exist constants $C_1$, $C_2$ and $r_0>0$ such that for all $x\in\R^d$,
\begin{equation}\label{pro2-3}L_{>1}\phi(x)\le -C_1\Lambda(x) \phi(x)\I_{B(0,r_0)^c}(x)+C_2\I_{B(0,r_0)}(x).\end{equation}
 \end{proposition}

\begin{proof} According to Proposition \ref{pro1}(2), we only need to verify \eqref{pro2-3} for $|x|$ large enough.
First, by the fact that for any $a$, $b\in\R$, $|a+b|^\gamma\le
|a|^\gamma+|b|^\gamma,$ we get that for any $x\in\R^d$,
\begin{equation*}
\begin{split} \frac{1}{2}\sum_{i=1}^d\int_{\{|z|>1\}}\big(\phi(x+ze_i)-\phi(x)\big)\frac{1}{|z|^{1+\alpha}}\,dz
&\le \frac{1}{2}\sum_{i=1}^d\int_{\{|z|>1\}}\frac{1}{|z|^{1+\alpha-\gamma}}\,dz\\
&=\frac{d}{\alpha-\gamma}
.\end{split}
\end{equation*}
On the other hand, for $|x|>
2^{1/\gamma}\sqrt{d} $ large enough, there is an integer $1\le
k\le d$ such that $|x_k|\ge |x|/\sqrt{d}$, and so
\begin{align*}  &\frac{1}{2}\sum_{i=1}^d\int_{\{|z|>1\}}\big(\phi(x+ze_i)-\phi(x)\big)\frac{e^{V(x)-V(x+ze_i)}}{|z|^{1+\alpha}}\,dz\\
&= \frac{e^{V(x)}}{2}\sum_{i=1}^d\int_{\{|z|>1\}}\frac{|x_i+z|^\gamma-|x_i|^\gamma}{|z|^{1+\alpha}}e^{-V(x+ze_i)}\,dz\\
&= \frac{e^{V(x)}}{2}\bigg[\sum_{i=1}^d\int_{\{|x_i+z|\le |x_i|, |z|>1\}}\frac{|x_i+z|^\gamma-|x_i|^\gamma}{|z|^{1+\alpha}}e^{-V(x+ze_i)}\,dz\\
&\qquad\qquad+ \sum_{i=1}^d\int_{\{|x_i+z|\ge |x_i|, |z|>1\}}\frac{|x_i+z|^\gamma-|x_i|^\gamma}{|z|^{1+\alpha}}e^{-V(x+ze_i)}\,dz\bigg]\\
&\le \frac{e^{V(x)}}{2}\bigg[\int_{\{|x_k+z|\le |x_k|, |z|>1\}}\frac{|x_k+z|^\gamma-|x_k|^\gamma}{|z|^{1+\alpha}}e^{-V(x+ze_k)}\,dz\\
&\qquad\qquad+ \sum_{i=1}^d\int_{\{|x_i+z|\ge |x_i|, |z|>1\}}\frac{|x_i+z|^\gamma-|x_i|^\gamma}{|z|^{1+\alpha}}e^{-V(x+ze_i)}\,dz\bigg],\end{align*}
where in the inequality above we have dropped the sum with $i\neq
k$, since it is negative. It is easy to see that the right hand side is dominated by
\begin{align*}
&\frac{e^{V(x)}}{2} \bigg[\int_{\{|x_k+z|\le 1, |z|>1\}}\frac{1-|x_k|^\gamma}{|z|^{1+\alpha}}e^{-V(x+ze_k)}\,dz\\
&\qquad \qquad +\int_{\{1<|x_k+z|\le |x_k|, |z|>1\}}\frac{|x_k+z|^\gamma-|x_k|^\gamma}{|z|^{1+\alpha}}e^{-V(x+ze_k)}\,dz\\
&\qquad\qquad+ \sum_{i=1}^d\int_{\{|x_i+z|\ge |x_i|, |z|>1\}}\frac{|x_i+z|^\gamma-|x_i|^\gamma}{|z|^{1+\alpha}}e^{-V(x+ze_i)}\,dz\bigg]\\
&\le -\frac{e^{V(x)}}{4} \left(\Big(\inf_{|u_k|\le 1} e^{-V(x_1,\cdots, x_{k-1}, u_k, x_{k+1},\cdots,x_d)}\Big)\int_{\{|x_k+z|\le 1\}}\frac{dz}{|z|^{1+\alpha}}\right)|x_k|^\gamma \\
&\quad+\frac{de^{V(x)}}{2}\Gamma^V_{\sup}(x)\int_{\{|z|>1\}}\frac{|z|^\gamma}{|z|^{1+\alpha}}\,dz,\end{align*}
where in the inequality above we have used the
fact that for any $x_k$, $z\in\R$ with $|x_k|\ge |x|/\sqrt{d}\ge 2^{1/\gamma}>2$ and $|x_k+z|\le 1,$ it holds $|z|\ge |x_k|-|x_k+z|>1$,
and we also have dropped the second term since it is negative too.  Furthermore, combining the fact $|x_k|\ge |x|/\sqrt{d}$ with
\eqref{pro2-1} and \eqref{thm1-0}, we find that the right hand side of the inequality above is smaller than
\begin{align*}
-c_1\frac{e^{V(x)}}{(1+|x|)^{1+\alpha}} \Gamma^V_{\inf}(x)\phi(x)+c_2e^{V(x)}\Gamma^V_{\sup}(x)\le -c_3\frac{e^{V(x)}}{(1+|x|)^{1+\alpha}} \Gamma^V_{\inf}(x)\phi(x),\end{align*} where $c_i(i=1,2,3)$ are positive constants.
Therefore, the desired assertion follows from all the estimates above.
 \end{proof}

Next, we turn to the local super Poincar\'{e} inequality for
$(D,\mathscr{D}(D))$. As mentioned in Section \ref{section1}, the
local super Poincar\'{e} inequality for $(D_S,\mathscr{D}(D_S))$ is
derived from the classical fractional Sobolev inequality, see
\cite[Lemma 3.1]{WW}. However, it seems that the  Dirichlet form
$(D,\mathscr{D}(D))$ does not have a connection
with the fractional Laplacian $-(-\Delta)^{\alpha/2}$. Thus, in the following proposition we need
a completely different approach, which is connected with the property of doubling measure in harmonic analysis, see \cite{St}.

\begin{proposition}\label{pro3} There exists a constant $C_3>0$ such that for any $f\in C_c^\infty(\R^d)$ and $r>0$,
\begin{equation}\label{pro3-local} \int_{B(0,r)}f^2(x)\,\mu_V(dx)\le t D(f,f)+\beta_r(t) \mu_V(|f|)^2,\quad t>0,\end{equation} where
$$\beta_r(t)=C_3\left[ t\wedge  \left(r^\alpha \frac{ \sup_{|x|\le 2\sqrt{d}r}e^{V(x)}}{\inf_{|x|\le r} e^{V(x)}}\right)\right]^{-d/\alpha} \frac{\big(\sup_{|x|\le 2\sqrt{d}r}e^{V(x)}\big)^{2+d/\alpha}}{\big(\inf_{|x|\le r} e^{V(x)}\big)^{1+d/\alpha}}.$$ In particular, for any $r_0>0$, there is a constant $C_4$ depending on $r_0$ such that for any $f\in C_c^\infty(\R^d)$ and $r\ge r_0$, the local Poincar\'{e} inequality \eqref{pro3-local} holds with
$$\beta_r(t)=C_4\left(1+t^{-d/\alpha}\right)\frac{\big(\sup_{|x|\le 2\sqrt{d}r}e^{V(x)}\big)^{2+d/\alpha}}{\big(\inf_{|x|\le r} e^{V(x)}\big)^{1+d/\alpha}}$$ for all $t>0$.
\end{proposition}
\begin{proof} The second assertion immediately follows from the first one, by the fact that the function $r\mapsto r^\alpha \frac{ \sup_{|x|\le 2\sqrt{d}r}e^{V(x)}}{\inf_{|x|\le r} e^{V(x)}}$ is increasing. Thus, we only need to prove the first one, which is split into three steps.

(1) For any $0<s\le r$ and $f\in C_c^\infty(\R^d)$, define
$$f_s(x):=\frac{1}{|B(0,s)|}\int_{B(x,s)}f(z)\,dz,\quad x\in B(0,r),$$ where $|B(0,s)|$ denotes the volume of the ball with center at 0 and radius $s$. We have
$$\sup_{x\in B(0,r)}|f_s(x)|\le \frac{1}{|B(0,s)|} \int_{B(0,2r)}|f(z)|\,dz,$$
and \begin{align*}\int_{B(0,r)}|f_s(x)|\,dx&\le \int_{B(0,r)}\frac{1}{|B(0,s)|}\int_{B(x,s)}|f(z)|\,dz\,dx\\
&\le \int_{B(0,2r)}\bigg(\frac{1}{|B(0,s)|}\int_{B(z,s)}\,dx\bigg)|f(z)|\,dz\\
&\le \int_{B(0,2r)}|f(z)|\,dz.
\end{align*} Thus,
$$\aligned\int_{B(0,r)}f_s^2(x)\,dx\le & \Big(\sup_{x\in B(0,r)}|f_s(x)|\Big) \int_{B(0,r)}|f_s(x)|\,dx\\
\le &\frac{1}{|B(0,s)|} \bigg(\int_{B(0,2r)}|f(z)|\,dz\bigg)^2.\endaligned$$

Furthermore, by the Cauchy-Schwarz inequality, for any $f\in C_c^\infty(\R^d)$ and $0<s\le r,$
\begin{align*}\int_{B(0,r)}&f^2(x)\,dx\\
\le & 2\int_{B(0,r)}\big(f(x)-f_s(x)\big)^2\,dx+ 2\int_{B(0,r)}f^2_s(x)\,dx\\
\le &2\int_{B(0,r)}\frac{1}{|B(0,s)|}\int_{B(x,s)}(f(x)-f(y))^2\,dy\,dx+ \frac{2}{|B(0,s)|} \bigg(\int_{B(0,2r)}|f(z)|\,dz\bigg)^2.\end{align*}
Using the convention that $(y_0-x_0)e_0=0$ and the inequality that
$$(a_1+\ldots+a_d)^2\le d(a_1^2+\ldots+a_d^2),\quad a_1,\ldots, a_d\in\R$$ deduced from the Cauchy-Schwarz inequality again, we find that the right hand side is dominated by \begin{align*}&\frac{2d}{|B(0,s)|}\sum_{i=1}^d\int_{B(0,r)}\int_{B(x,s)}\Big(f\big(x+(y_1-x_1)e_1+\ldots+(y_i-x_i)e_i\big)\\
&\quad\qquad\qquad\qquad-
f\big(x+(y_1-x_1)e_1+\ldots+(y_{i-1}-x_{i-1})e_{i-1}\big)\Big)^2\,dy\,dx\\
&\quad+ \frac{2}{|B(0,s)|} \bigg(\int_{B(0,2r)}|f(z)|\,dz\bigg)^2\\
&\le\frac{2d}{|B(0,s)|}\sum_{i=1}^d\int_{B(0,r)}\int_{\{|z_1|\le s\}}\ldots\int_{\{|z_d|\le s\}}\Big(f\big(x+z_1e_1+\ldots+z_ie_i\big)\\
&\quad\qquad\qquad\qquad-
f\big(x+z_1e_1+\ldots+z_{i-1}e_{i-1}\big)\Big)^2\,dz_1\ldots dz_d\,dx\\
&\quad+ \frac{2}{|B(0,s)|} \bigg(\int_{B(0,2r)}|f(z)|\,dz\bigg)^2\\
&\le\frac{2^{d}ds^{d-1}}{|B(0,s)|}\sum_{i=1}^d\int_{B(0,\sqrt{d}(r+s))}\int_{\{|z_i|\le s\}}\Big(f\big(x+z_ie_i\big)-f(x)\Big)^2\,dz_i\,dx\\
&\quad+ \frac{2}{|B(0,s)|} \bigg(\int_{B(0,2r)}|f(z)|\,dz\bigg)^2,\end{align*}
where the two inequalities above follow from setting $z_k=y_k-z_k$ for all $1\le k \le i-1$ and enlarging the domain of $x$ respectively.
On the other hand, we can easily see that the right hand side of the inequality above is smaller than
\begin{align*}
 &\frac{2^{d}ds^{d+\alpha}}{|B(0,s)|}\sum_{i=1}^d\int_{B(0,\sqrt{d}(r+s))}\int_{\{|z|\le s\}}\frac{\Big(f\big(x+ze_i\big)-f(x)\Big)^2}{|z|^{1+\alpha}}\,dz\,dx\\
&+ \frac{2}{|B(0,s)|} \bigg(\int_{B(0,2r)}|f(z)|\,dz\bigg)^2.\end{align*} Therefore, for any $f\in C_c^\infty(\R^d)$ and $0<s\le r$,
\begin{equation*}\begin{split} \int_{B(0,r)}f^2(x)\,dx
 &\le\frac{2^{d}ds^{d+\alpha}}{|B(0,s)|}\sum_{i=1}^d\int_{B(0,\sqrt{d}(r+s))}\int_{\{|z|\le s\}}\frac{\Big(f\big(x+ze_i\big)-f(x)\Big)^2}{|z|^{1+\alpha}}\,dz\,dx\\
&\quad + \frac{2}{|B(0,s)|} \bigg(\int_{B(0,2r)}|f(z)|\,dz\bigg)^2.\end{split} \end{equation*}

(2) According to the inequality above, for any $f\in C_c^\infty(\R^d)$ and $0<s\le r,$
{\small \begin{align*}\int_{B(0,r)}&f^2(x)\,\mu_V(dx)\\
\le & \frac{1}{\inf_{|x|\le r} e^{V(x)}}\int_{B(0,r)}f^2(x)\,dx\\
\le & \Bigg[\frac{2^{d}ds^{d+\alpha}}{|B(0,s)|\big(\inf_{|x|\le r} e^{V(x)}\big)}\Bigg] \sum_{i=1}^d\int_{B(0,\sqrt{d}(r+s))}\int_{\{|z|\le s\}}\frac{\Big(f\big(x+ze_i\big)-f(x)\Big)^2}{|z|^{1+\alpha}}\,dz\,dx\\
&+ \frac{2}{|B(0,s)|\big(\inf_{|x|\le r} e^{V(x)}\big)} \bigg(\int_{B(0,2r)}|f(z)|\,dz\bigg)^2\\
\le & \Bigg[\frac{2^{d}ds^{d+\alpha}\big(\sup_{|x|\le 2\sqrt{d}r}e^{V(x)}\big)}{|B(0,s)|\big(\inf_{|x|\le r} e^{V(x)}\big)}\Bigg] \\
&\qquad \qquad \times \sum_{i=1}^d\int_{B(0,2\sqrt{d}r)}\int_{\{|z|\le s\}}\frac{\Big(f\big(x+ze_i\big)-f(x)\Big)^2}{|z|^{1+\alpha}}\,dz\,\mu_V(dx)\\
&+ \frac{2\big(\sup_{|x|\le 2r}e^{V(x)}\big)^2}{|B(0,s)|\big(\inf_{|x|\le r} e^{V(x)}\big)} \bigg(\int_{B(0,2r)}|f(x)|\,\mu_V(dx)\bigg)^2\\
\le &\Bigg[\frac{2^{d+1}ds^{d+\alpha}\big(\sup_{|x|\le 2\sqrt{d}r}e^{V(x)}\big)}{|B(0,s)|\big(\inf_{|x|\le r} e^{V(x)}\big)}\Bigg] D(f,f)+ \frac{2\big(\sup_{|x|\le 2r}e^{V(x)}\big)^2}{|B(0,s)|\big(\inf_{|x|\le r} e^{V(x)}\big)} \mu_V(|f|)^2,\end{align*}}
which implies that the local super Poincar\'{e} inequality \eqref{pro3-local} holds with
{\small $$\beta_r(t)=\inf\left\{\frac{2\big(\sup_{|x|\le 2r}e^{V(x)}\big)^2}{|B(0,s)|\big(\inf_{|x|\le r} e^{V(x)}\big)}: 0<s\le r\textrm{ and } \frac{2^{d+1}ds^{d+\alpha}\big(\sup_{|x|\le 2\sqrt{d}r}e^{V(x)}\big)}{|B(0,s)|\big(\inf_{|x|\le r} e^{V(x)}\big)}\le t   \right\}$$ } for any $t>0$.

(3) Next, we fix $r>0$. If $0<t\le \frac{2^{d+1}d r^{d+\alpha}\big(\sup_{|x|\le 2\sqrt{d}r}e^{V(x)}\big)}{|B(0,r)|\big(\inf_{|x|\le r} e^{V(x)}\big)}$, then one can choose $s\in (0,r]$ such that $\frac{2^{d+1}ds^{d+\alpha}\big(\sup_{|x|\le 2\sqrt{d}r}e^{V(x)}\big)}{|B(0,s)|\big(\inf_{|x|\le r} e^{V(x)}\big)}= t,$ and so there is a constant $C_5>0$ (independent of $r, t$) such that $ \beta_r(t)\le C_5 t^{-d/\alpha}\frac{\big(\sup_{|x|\le 2\sqrt{d}r}e^{V(x)}\big)^{2+d/\alpha}}{\big(\inf_{|x|\le r} e^{V(x)}\big)^{1+d/\alpha}} .$

If $t\ge \frac{2^{d+1}d r^{d+\alpha}\big(\sup_{|x|\le 2\sqrt{d}r}e^{V(x)}\big)}{|B(0,r)|\big(\inf_{|x|\le r} e^{V(x)}\big)}$, then, by taking $s=r$ in the right hand side of the definition of $\beta_r(t)$ above, we find that $\beta_r(t)\le C_6\frac{\big(\sup_{|x|\le 2r}e^{V(x)}\big)^2}{r^d\big(\inf_{|x|\le r} e^{V(x)}\big)}.$

Combining with both estimates above, we complete the proof.

 \end{proof}

\section{Proofs and Complements}\label{section3}
\subsection{Proofs of Theorem \ref{thm1}, Corollaries and Example}
We begin with the proof of Theorem \ref{thm1}.

\begin{proof}[{\bf Proof of Theorem \ref{thm1}}]
Since $C_c^\infty(\R^d)$ is the core of $\mathscr{D}(D)$, it is
enough to prove the desired Poincar\'{e} type inequalities for any $f\in
C_c^\infty(\R^d)$. According to Proposition \ref{pro2}
 and $\phi\ge1$, there exists a
constant $r_0>0$ such that
$$\I_{B(0,r)^c}\le \frac{1}{C_1\Phi(r)}\frac{-L_{>1}\phi}{\phi}+\frac{C_2}{C_1\Phi(r)}\I_{B(0,r_0)}
,\quad r\ge r_0,$$where $$\Phi(r)=\inf_{|x|\ge r} \Lambda(x).$$ Then, for any $f\in
C_c^\infty(\R^d)$,
$$\mu_V(f^2\I_{B(0,r)^c})\le
\frac{1}{C_1\Phi(r)}\mu_V\left(f^2\frac{-L_{>1}\phi}{\phi}\right)+\frac{C_2}{C_1\Phi(r)}\mu_V(f^2\I_{B(0,r_0)}),\quad
r\ge r_0.$$ We note that for any $x$, $y\in\R^d$,
\begin{equation*}
\begin{split}
\left(\frac{f^2(x)}{\phi(x)}-\frac{f^2(y)}{\phi(y)}\right)(\phi(x)-\phi(y))
&=f^2(x)+f^2(y)-\left(\frac{\phi(y)}{\phi(x)}f^2(x)+\frac{\phi(x)}{\phi(y)}f^2(y)\right)\\
&\le f^2(x)+f^2(y)-2|f(x)||f(y)|\\
&\le (f(x)-f(y))^2,
\end{split}
\end{equation*} which, along with Proposition \ref{pro1}(2), yields that
$$\mu_V\left(f^2\frac{-L_{>1}\phi}{\phi}\right)\le
D_{>1}(f,f)\le D(f,f).$$ Therefore, for any $f\in C_c^\infty(\R^d)$,
\begin{equation}\label{proof1} \mu_V(f^2\I_{B(0,r)^c})\le
\frac{1}{C_1\Phi(r)}D(f,f)+\frac{C_2}{C_1\Phi(r)}\mu_V(f^2\I_{B(0,r_0)}),\quad
r\ge r_0.\end{equation}

On the other hand, by Proposition \ref{pro3}, there exists a constant $C_3>0$ such that for any $f\in
C_c^\infty(\R^d)$, $r\ge r_0$ and $t>0$,
\begin{equation*}
\begin{split}
\mu_V(f^2\I_{B(0,r)}) \le &t{D}(f,f)+\beta_r(t)\mu_V(|f|)^2,
\end{split}
\end{equation*} where
$$\beta_r(t)=C_3(1+t^{-d/\alpha})\frac{\big(\sup_{|x|\le 2\sqrt{d}r}e^{V(x)}\big)^{2+d/\alpha}}{\big(\inf_{|x|\le r} e^{V(x)}\big)^{1+d/\alpha}}.$$

Combining it with \eqref{proof1}, we get that for any $f\in
C_c^\infty(\R^d)$, $r\ge r_0$ and $t>0$,
\begin{equation}\label{proof2}
\begin{split}\mu_V(f^2)=&\mu_V(f^2\I_{B(0,r)^c})+\mu_V(f^2\I_{B(0,r)})\\
\le &\left(t+\frac{1+C_2 t}{C_1\Phi(r)}\right)D(f,f)+
\beta_r(t)\left(1+\frac{C_2}{C_1\Phi(r)}\right)\mu_V(|f|)^2.\end{split}
\end{equation}

(1) Suppose that $$\lim\limits_{r\to\infty} \Phi(r)>0.$$ Then, for
any fixed $t>0$, \eqref{proof2} is just the defective Poincar\'{e}
inequality. Since the conservative and symmetric Dirichlet form $(D,
\mathscr{D}(D))$ is irreducible, i.e.\ $D(f, f)=0$ implies $f$ is a
constant function, it follows from \cite[Corollary 1.2]{W13} (see
also \cite[Theorem 1]{Mi}) that the defective Poincar\'{e}
inequality \eqref{proof2} implies the true Poincar\'{e}
inequality \eqref{thm1-1}.

(2) Now, we assume that
$$\lim\limits_{r\to\infty} \Phi(r)=\infty.$$ For any $s>0$, taking $t=s/2$ and $r=\Phi^{-1}((2+C_2s)/(C_1s))$ in \eqref{proof2}, we can get the super Poincar\'{e} inequality \eqref{thm1-2} with the desired rate function $\beta$ (possibly with different positive constants $C_1$ and $C_2$).
\end{proof}

\begin{proof}[{\bf Proof of Corollary \ref{exm1}}] Let $\varepsilon_*=\min_{1\le i\le d} \varepsilon_i$. We have
\begin{equation*}
\begin{split}\Gamma_{\inf}^V(x)\ge c_1 e^{-V(x)}(1+|x|)^{1+\varepsilon_*}, \,\, \Gamma_{\sup}^V(x)=e^{-V(x)}\textrm{ and } \Lambda(x)\ge c_2 (1+|x|)^{\varepsilon_*-\alpha}\end{split}
\end{equation*}for some constants $c_1,c_2>0$. Then, \eqref{thm1-0} is satisfied for any $\gamma>0$, if $\varepsilon_*\ge \alpha$. Next, we will prove the desired assertions.

(1) If $\varepsilon_i\ge \alpha$ for all $1\le i\le d$, then the Poincar\'{e} inequality \eqref{thm1-1} follows from
Theorem \ref{thm1}(1). Assume that for
some $1\le i\le d$, $\varepsilon_i\in (0,\alpha)$. To disprove the Poincar\'{e} inequality \eqref{thm1-1} in this case, let us consider the function
$f(x)=g(x_i)$, where $g\in C_c^\infty(\R)$.  By \cite[Theorem 2.1(1)]{CW13}, in the present setting $C_1^b(\R^d)\subset \mathscr{D}(D)$, and so we can apply $f\in C_b^\infty(\R^d)\subset \mathscr{D}(D)$ into the Poincar\'{e} inequality
\eqref{thm1-1}. Furthermore, it is easy to see that for this class of functions, the Poincar\'{e} inequality \eqref{thm1-1} is reduced into
\begin{equation}\label{proof-exm1-1}m(g^2)\le C\int_{\R\times \R} \frac{(g(x+z)-g(x))^2}{|z|^{1+\alpha}}\,dz\,m(dx),\quad m(g)=0, g\in C_c^\infty(\R),\end{equation} where $m(dx)=C_{\varepsilon_i}(1+|x|)^{-(1+\varepsilon_i)}\,dx.$
According to \cite[Corollary 1.2]{WW}, we know that for
$\varepsilon_i\in(0,\alpha)$, the inequality \eqref{proof-exm1-1}
does not hold. Therefore, for any constant $C>0$, the Poincar\'{e}
inequality \eqref{thm1-1} also does not hold.

(2) Let $\varepsilon_i>\alpha$ for all $1\le i\le d$.
It is easy to see that $\Phi(r)\ge c_3 r^{\varepsilon_*-\alpha}$ for $r$ large enough.
Then, $ \Phi^{-1}(r^{-1})\le c_4r^{-1/(\varepsilon_*-\alpha)}$ for $r$ small enough, so that
$$\beta(r)\le c_5\left(1+r^{-(\frac{d}{\alpha}+\frac{(2\alpha+d)\sum_{i=1}^d(1+\varepsilon_i)}
{\alpha(\varepsilon_*-\alpha)})}\right),\quad r>0$$
for some constant $c_5>0$. The equivalence of the super Poincar\'{e} inequality and the corresponding bound of $\|P_t\|_{L^1(\R^d;\mu_V)\to L^\infty(\R^d;\mu_V)}$ then follows from \cite[Theorem 3.3.15(2)]{WBook}.

Next, we prove that if $\varepsilon_i\in(0,\alpha]$ for some $1\le i\le d$, then for any $\beta:(0,\infty)\to(0,\infty)$ the super Poincar\'{e} inequality \eqref{thm1-2} does not hold. Indeed, if the inequality \eqref{thm1-2} holds, then, applying the function
$f(x)=g(x_i)$ (where $g\in C_c^\infty(\R)$) mentioned above, we have
{\small\begin{equation}\label{proof-exm1-2}
\begin{split}m(g^2)\le & r\int_{\R\times \R} \frac{(g(x+z)-g(x))^2}{|z|^{1+\alpha}}\,dz\,m(dx)+\beta(r)m(|g|)^2,\quad r>0, g\in
C_c^\infty(\R).\end{split}
\end{equation} }However, according to \cite[Corollary 1.2(2)]{WW}, \eqref{proof-exm1-2} can not be true.
\end{proof}

\begin{proof}[{\bf Proof of Corollary \ref{exm2}}] Suppose that $\varepsilon_*=\min_{1\le i\le d} \varepsilon_i\ge0$. Then, there are constants $c_i>0$ $(i=1,2)$ such that for all $x\in \R^d$,
\begin{equation*}
\begin{split}\Gamma_{\inf}^V(x)\ge c_1 e^{-V(x)}(1+|x|)^{1+\alpha}\log^{\varepsilon_*}(e+|x|),\,\,
\Gamma_{\sup}^V(x)=e^{-V(x)}\end{split}
\end{equation*} and $\Lambda(x)\ge c_2 \log^{\varepsilon_*}(e+|x|)$. It is obvious that \eqref{thm1-0} holds.

(1) If $\varepsilon_i\ge 0$ for all $1\le i\le d$, then the Poincar\'{e} inequality \eqref{thm1-1} follows from
Theorem \ref{thm1}(1).
To prove that the Poincar\'{e} inequality \eqref{thm1-1} does not hold when
$\varepsilon_i<0$ for some $1\le i\le d$, we only need to consider the Poincar\'{e} inequality \eqref{proof-exm1-1}, where $m(dx)=C_{\varepsilon_i}(1+|x|)^{-(1+\alpha)}\log^{-\varepsilon_i}(e+|x|)\,dx$. Then, according to \cite[Corollary 1.3]{WW}, \eqref{proof-exm1-1} does not hold, and so \eqref{thm1-1} does not too.

(2) Let $\varepsilon_i>0$ for all $1\le i\le d$. It is easy to see that $\Phi(r)\ge c_3\log^{\varepsilon_*}r$ for $r$ large enough. Then, $ \Phi^{-1}(1/r)\le c_4\exp(c_4 r^{-1/\varepsilon_*})$ for $r$ small enough, so that
$$\beta(r)\le \exp\Big( c_5(1+r^{-1/\varepsilon_*})\Big),\quad r>0$$ for some constant $c_5>0$.
When $\varepsilon_i>1$ for all $1\le i\le d$, the equivalence of the super Poincar\'{e} inequality and the corresponding bound of $\|P_t\|_{L^1(\R^d; \mu_V)\to L^\infty(\R^d;\mu_V)}$ follows from \cite[Theorem 3.3.15(1)]{WBook}.

Next, we prove that if $\varepsilon_i\le 0$ for some $1\le i\le d$,
then for any $\beta:(0,\infty)\to(0,\infty)$ the super Poincar\'{e}
inequality \eqref{thm1-2} does not hold. Indeed, in this case, we
only need to prove that the super Poincar\'{e} inequality
\eqref{proof-exm1-2} does not hold for
$m(dx)=C_{\varepsilon_i}(1+|x|)^{-(1+\alpha)}\log^{-\varepsilon_i}(e+|x|)\,dx$.
This is just a consequence of \cite[Corollary 1.3]{WW}.

According to \cite[Corollary 3.3.4(1)]{WBook}, the super
Poincar\'{e} inequality with $\beta(r)=\exp(c(1+r^{-1}))$ for some
$c>0$ is equivalent to the log-Sobolev inequality for some constant
$C>0$, and so according to the conclusions above we can conclude the
last assertion in (2) for log-Sobolev inequality.
\end{proof}

\begin{proof}[{\bf Proof of Example \ref{coro}}] We first estimate $\Gamma_{\inf}^V(x)$ and $\Gamma_{\sup}^V(x)$ respectively. On the one hand, for all $x\in \R^d$, by using the boundness of $a_i$ for all $1\le i\le d$,
 \begin{align*}\Gamma_{\inf}^V(x)\ge &c_1\inf_{1\le i\le d:|x_i|\ge|x|/\sqrt{d}}\inf_{|u_i|\le 1}\bigg[(1+|u_i|)^{-1-a_i^*(x|u_i)} \prod_{j\neq i}(1+|x_j|)^{-1-a_j^*(x|u_i)}\bigg]\\
\ge &c_2\inf_{1\le i\le d:|x_i|\ge|x|/\sqrt{d}}\inf_{|u_i|\le 1}\bigg[\prod_{j\neq i}(1+|x_j|)^{-1-a_j^*(x|u_i)}\bigg]\\
\ge &c_2\inf_{1\le i\le d:|x_i|\ge|x|/\sqrt{d}}\bigg[\prod_{j\neq i}(1+|x_j|)^{-1-\sup_{|u_i|\le 1}a_j^*(x|u_i)}\bigg]\\
=&c_2 \inf_{1\le i\le d:|x_i|\ge|x|/\sqrt{d}}\bigg[(1+|x_i|)^{1+\sup_{|u_i|\le 1}a_i^*(x|u_i)}\prod_{j=1}^d (1+|x_j|)^{-1-\sup_{|u_i|\le 1}a_j^*(x|u_i)}\bigg] \\
\ge&c_3  (1+|x|)^{1+A(x)}\prod_{j=1}^d (1+|x_j|)^{-1-N_j(x)}.\end{align*}
On the other hand, for all $x\in \R^d$, \begin{equation*}
\begin{split}\Gamma_{\sup}^V(x)\le&c_4\sup_{1\le i\le d}\sup_{|u_i|\ge |x_i|}\bigg[(1+|u_i|)^{-1-a_i^*(x|u_i)} \prod_{j\neq i}(1+|x_j|)^{-1-a_j^*(x|u_i)}\bigg]\\
\le& c_4\sup_{1\le i\le d}\prod_{1\le j\le d}(1+|x_j|)^{-1-\inf_{|u_i|\ge |x_i|} a_j^*(x|u_i)}\\
\le &c_4 \prod_{1\le j\le d} (1+|x_j|)^{-1-M_j(x)}.\end{split}
\end{equation*} Under \eqref{coro-22} and \eqref{coro-1}, \eqref{thm1-0} holds true for any $\gamma>0$.  Furthermore, \eqref{coro-22} implies that for $|x|$ large enough and all $1\le j\le d$, $a_j(x)\ge N_j(x),$ so for $|x|$ large enough,
$$\Lambda(x)\ge c_5(1+|x|)^{A(x)-\alpha}.$$
Having these estimates at hand, we can obtain the required assertions by following the arguments of Corollary \ref{exm1} and using Theorem \ref{thm1}.   \end{proof}

\subsection{Complement: Entropy Inequalities and Tensorisation Property for Singular Stable-like Dirichlet Forms.}
In this part, we are concerned with the case that the reference measure $\mu_V$ is a
product measure on $\R^d$, and aim to consider entropy inequalities for the Dirichlet
form $(D,\mathscr{D}(D))$. Note that, the relative entropy
$\ent_\mu$ is defined on $L^1(\R^d;\mu)$ as follows
$$\ent_\mu(f):=\mu(f \log f)-\mu(f)\log \mu(f),\quad f>0.$$ The following theorem is a direct application of \cite[Theorem 1.5]{WJ13} and the sub-additivity property of the relative entropy.

\begin{theorem}\label{thm10}  Let $\mu_V=\mu_1\otimes\ldots\otimes\mu_d$ be a
product measure on $\R^d$ such that for any $1\le i\le d$, $\mu_i(dx_i)=e^{-V_i(x_i)}\,dx_i$
is a probability measure on $\R$, where $V_i$ is a Borel measurable
function on $\R$ and $e^{-V_i(\cdot)}$ may be unbounded. If for any $1\le i\le d$ there exists a constant $C_i>0$
such that for any $x,y\in\R$
\begin{equation}\label{thm10-1}\frac{e^{V_i(x)}+e^{V_i(y)}}{|x-y|^{1+\alpha}}\,\ge C_i,\end{equation} then
the following entropy inequality holds
\begin{equation}\label{thm10-2} \ent_{\mu_V}(f)\le 2\big(\sup_{1\le i\le
d}C_i^{-1}\big) D(f,\log f),\quad f\in \mathscr{D}(D),
f>0.\end{equation} In particular, the Poincar\'{e} inequality
\eqref{thm1-1} also holds.
\end{theorem}

\begin{proof} For the sake of completeness, we provide the details here. For any $f\in \mathscr{D}(D)$ with $f>0$ and for any
$x\in\R^d$, define $$f_{x,i}(y_i):=f(x_1,\ldots, x_{i-1}, y_i, x_{i+1},
\ldots,x_d),\quad y_i\in\R, 1\le i\le d.$$ By the sub-additivity property of the relative entropy (see
\cite[Proposition 4.1]{Le} or \cite[Corollary 3]{La}), for any $f\in
\mathscr{D}(D)$ with $f>0$,
$$\ent_{\mu_V}(f)\le \sum_{i=1}^d\int \ent_{\mu_i}f_{x,i}(y_i)\,\prod_{j\neq i}\mu_j(dx_j).$$ According to
\eqref{thm10-1} and \cite[Theorem 1.5]{WJ13}, for each $i$ the inner
relative entropy is at most
$$C_i^{-1}\int_{\R\times \R}\frac{\big(f(x+ze_i)-f(x)\big)\big(\log
f(x+ze_i)-\log f(x)\big)}{|z|^{1+\alpha}}\,dz\,\mu_i(dx_i).$$

Indeed, write $f_{x,i}$ as $f_{i}$ for simplicity. By the Jensen
inequality,
\begin{equation}\label{ent}
\begin{split}  \textrm{Ent}_{\mu_i}(f_i)&=\mu_i(f_i\log f_i)-\mu_i(f_i)\log \mu_i(f_i)\\
&\le \mu_i(f_i\log f_i)-\mu_i(f_i)\mu_i(\log f_i)\\
 &= \frac{1}{2}\int \bigg[f_i(x_i)\log f_i(x_i)+f_i(y_i)\log f_i(y_i)\\
&\qquad\qquad-f_i(x_i)\log f_i(y_i)- f_i(y_i)\log f_i(x_i)\bigg]\,\mu_i(dy_i)\,\mu_i(dx_i)\\
&= \frac{1}{2}\int (f_i(x_i)-f_i(y_i))(\log f_i(x_i)-\log
f_i(y_i))\,\mu_i(dy_i)\,\mu_i(dx_i).\end{split}\end{equation} On the
other hand, by \eqref{thm10-1}, we find that
$$\aligned
&\frac{1}{2} \int \big(f_i(x_i)-f_i(y_i)\big)\big(\log f_i(x_i)-\log f_i(y_i)\big)\,\mu_i(dy_i)\,\mu_i(dx_i)\\
&\le C_i^{-1}\int_{x_i\neq y_i}\frac{
\big(f_i(x_i)-f_i(y_i)\big)\big(\log f_i(x_i)-\log
f_i(y_i)\big)}{|x_i-y_i|^{1+\alpha}}
\,\frac{e^{-V_i(x_i)}+e^{-V_i(y_i)}}{2}\,dy_i\,dx_i\\
&= C_i^{-1} \int_{\R\times
\R}\frac{\big(f(x+ze_i)-f(x)\big)\big(\log
f(x+ze_i)-\log f(x)\big)}{|z|^{1+\alpha}}\,dz\,\mu_i(dx_i),
\endaligned$$ which, along with \eqref{ent}, yields the above desired assertion.

Summing up the conclusions above, we prove the required
assertion for the entropy inequality \eqref{thm10-2}.
The last conclusion follows from the well known fact that the entropy
inequality \eqref{thm10-2} is stronger than the Poincar\'{e}
inequality \eqref{thm1-1}. (To see this, one can apply \eqref{thm10-2} to the
function $1 + \varepsilon f$ and then take the limit as $\varepsilon\rightarrow0.$)
\end{proof}

The proof of Theorem \ref{thm10} is based on the sub-additivity property of entropy and the characterization of Dirichlet form
$(D,\mathscr{D}(D))$. Sub-additivity formulas are well known for variance and entropy, and then have been extended to $\Phi$-entropies in \cite[Proposition 3.1]{Cha}, which is called the tensorisation property. The tensorisation property for $\Phi$-entropies also can be used to establish $\Phi$-Sobolev inequalities for Dirichlet form
$(D,\mathscr{D}(D))$.  In particular, by tensorisation property the $\Phi$-Sobolev
inequalities are then infinite dimensional since they hold on the
product space with the maximum of the one dimensional constants, e.g.\ see \eqref{thm10-2}. However, such statements do not hold for weak/super Poincar\'{e} inequalities and transportation-cost inequalities. See \cite[Theorem 5]{BCR}, \cite[Section 1.3]{G} and \cite[Chapter 22]{Villani} for more details.


To show that Theorem \ref{thm10} is sharp, we consider the following
corollary, which is regarded as a continuation of Corollary \ref{exm1}.
\begin{corollary}\label{exm22}Let $$e^{-V(x)}=C_{\varepsilon_1, \ldots, \varepsilon_d}\prod_{i=1}^d(1+|x_i|)^{-(1+\varepsilon_i)}$$ with $\varepsilon_i>0$ for all $1\le i\le d$.
Then the entropy inequality \eqref{thm10-2} holds for some constant $C>0$
if and only if $\varepsilon_i\ge \alpha$ for all $1\le i\le
d$.\end{corollary}
\begin{proof} If $\varepsilon_i\ge \alpha$ for all $1\le i\le d$,
then \eqref{thm10-1} holds (see \cite[Example 1.6]{WJ13}), and so
the desired entropy inequality \eqref{thm10-2} follows from Theorem
\ref{thm10}. On the other hand, since the entropy
inequality \eqref{thm10-2} is stronger than the Poincar\'{e}
inequality \eqref{thm1-1}, by Corollary \ref{exm1}(1),
we know that \eqref{thm10-2} does not hold when $\varepsilon_i\in(0,
\alpha)$ for some $1\le i\le d.$
\end{proof}

\subsection{Complement: Weak Poincar\'{e} Inequalities for Singular
Stable-like Dirichlet Forms}

In the end of this section, we turn to the weak Poincar\'{e} inequality, which can be used
to characterize various convergence rates of the associated semigroups slower than exponential. In the following, let $\mu_{V_0}(dx)=e^{-V_0(x)}\,dx$ be a
probability measure on $\R^d$ such that $e^{-V_0(x)}$ is a bounded measurable function, and \eqref{thm1-0} and
$\lim_{r\to0}\Phi(r)>0$ hold with ${V_0}$ in place of ${V}$. Then,
for a probability measure $\mu_{V}(dx)=e^{-V(x)}\,dx$, we have
\begin{theorem}\label{thm2} Suppose that \begin{equation}\label{thm2-1} \sup_{x\in\R^d}{e^{V(x)-V_0(x)}}<\infty.\end{equation} If there exist a family of Borel sets $\{A_s\}_{s\ge0}$ such that $A_s\uparrow \R^d$ as $s\to\infty$, $\Phi_0(s):=\sup_{x\in A_s} { e^{V_0(x)-V(x)}}<\infty$ for any $s>0$ and $\lim_{s\to\infty}\Phi_0(s)=\infty$,  then the following
weak Poincar\'{e} inequality
\begin{equation}\label{thm2-2}\mu_V(f^2)\le \eta(r)
D(f,f)+r\|f\|_\infty^2, \quad r>0, f\in \mathscr{D}(D), \mu_V(f)=0
\end{equation} holds for
$$\eta(r)=C\inf\Big\{{\Phi_0(s)}: s>0\textrm{ such that }\mu_V(A_s)\ge \frac{1}{1+r}\Big\}$$ with some constant $C>0$ independent of $r$.

In particular, under \eqref{thm2-1} the weak Poincar\'{e} inequality \eqref{thm2-2} holds with $$\eta(r)=C_1\inf\Big\{{s}: s>0\textrm{ such that }\mu_V(D_s)\ge \frac{1}{1+r}\Big\},$$ where $D_s=\{x\in \R^d: e^{V_0(x)-V(x)}\le s\}$ and $C_1>0$ is independent of $r$.   \end{theorem}
\begin{proof} Since $ C_c^\infty(\R^d)$ is a core of $\mathscr{D}(D)$, we only need to consider $f\in C_c^\infty(\R^d)$. According to Theorem \ref{thm1}(1), the
following Poincar\'{e} inequality $$ \mu_{V_0}(f^2)\le
c_1\sum_{i=1}^d\int_{\R^d\times \R}
\frac{(f(x+ze_i)-f(x))^2}{|z|^{1+\alpha}}\,dz\,\mu_{V_0}(dx),\,\, f\in C_c^\infty(\R^d), \mu_{V_0}(f)=0$$ holds for
some constant $c_1>0$. Let $\{A_s\}_{s\ge0}$ be a family of subsets as in theorem. Therefore, for any $s>0$ and $f\in
C_c^\infty(\R^d)$,
\begin{align*}
&\int_{A_s}\left(f(x)-\frac{1}{\mu_V(A_s)}\int_{A_s}f(x)\,\mu_V(dx)\right)^2\,\mu_V(dx)\\
&=\inf_{a\in\R}\int_{A_s}(f(x)-a)^2\,\mu_V(dx)\\
&\le \int_{A_s}(f(x)-\mu_{V_0}(f))^2\,\mu_V(dx)\\
&\le \left(\sup_{x\in A_s}\frac{e^{-V(x)}}{e^{-V_0(x)}}\right)\int_{A_s}(f(x)-\mu_{V_0}(f))^2\,\mu_{V_0}(dx)\\
&\le c_1\left(\sup_{x\in A_s}\frac{e^{-V(x)}}{e^{-V_0(x)}}\right)\sum_{i=1}^d\int_{\R^d\times \R}
\frac{(f(x+ze_i)-f(x))^2}{|z|^{1+\alpha}}\,dz\,\mu_{V_0}(dx)\\
&\le c_2\Phi_0(s)\sum_{i=1}^d\int_{\R^d\times \R}
\frac{(f(x+ze_i)-f(x))^2}{|z|^{1+\alpha}}\,dz\,\mu_V(dx),
\end{align*} where in the last inequality we have used \eqref{thm2-1}. Then the required weak Poincar\'{e} inequality
\eqref{thm2-2} follows from \cite[Theorem 4.3.1]{WBook}.

Taking $A_s=D_s$ and using the fact $\Phi_0(s)=\sup_{x\in D_s} e^{V_0(x)-V(x)}\le s$, one can get the second assertion.
\end{proof}

We consider the following corollary to illustrate the power of Theorem \ref{thm2}.
\begin{corollary}\label{exm3}{\rm [Continuation of Corollaries \ref{exm1} and \ref{exm2}]}\ \
\begin{itemize}
\item[(1)] Let $$e^{-V(x)}=C_{\varepsilon_1, \ldots, \varepsilon_d}\prod_{i=1}^d(1+|x_i|)^{-(1+\varepsilon_i)}$$
with $0<\varepsilon_*:=\min_{1\le i\le
d}\varepsilon_i<\alpha$. Then the weak
Poincar\'{e} inequality \eqref{thm2-2} holds with
\begin{equation}\label{exm3-1}\eta(r)=c_1\Big(1+r^{-{\sum_{i=1}^d(\alpha-\varepsilon_i)^+}/{\varepsilon_*}}\Big),\quad r>0\end{equation}
for some constant $c_1>0$. Consequently, 
there exists a constant $\lambda>0$ such that
\begin{equation*}\begin{split}  \|P_t-\mu_V\|_{L^\infty(\R^d;\mu_V)\to L^2(\R^d;\mu_V)}:&=\sup_{f\in L^\infty(\R^d;\mu_V)}\|P_tf-\mu_V(f)\|_{L^2(\R^d;\mu_V)}\\
&\le
\lambda
t^{-{\varepsilon_*}/{\sum_{i=1}^d(\alpha-\varepsilon_i)^+}},\quad
t>0.\end{split}\end{equation*}
In particular, let $e^{-V(x)}$ be the density function  above with $\varepsilon_k\in(0,\alpha)$ for some  $1\le k\le d$ and $\varepsilon_i\in[\alpha,\infty)$ for any $i\neq k$. Then, the rate function $\alpha$ given by \eqref{exm3-1} is
$$\eta(r)=c_3\Big(1+r^{-{(\alpha-\varepsilon_k)}/{\varepsilon_k}}\Big),\quad r>0$$ for some constant $c_3>0$, which
 is sharp in the sense
that \eqref{thm2-2} does not hold  if
$$\lim_{r\to0}r^{(\alpha-\varepsilon_k)/\varepsilon_k}\eta(r)=0.$$
    \item[(2)] Let $$e^{-V(x)}=C_{\varepsilon_1, \ldots, \varepsilon_d,\alpha}\prod_{i=1}^d(1+|x_i|)^{-(1+\alpha)}\log^{-\varepsilon_i}(e+|x_i|)$$ with 
        $\varepsilon_*:=\min_{1\le i\le
d}\varepsilon_i<0$. Then the weak Poincar\'{e}
inequality \eqref{thm2-2} holds with
\begin{equation}\label{exm3-2}\eta(r)=c\Big(1+\log^{-\sum_{i=1}^d(\varepsilon_i\wedge0)}(1+r^{-1})\Big),\quad r>0\end{equation} for some
constant $c>0$. Consequently, there exist constants $\lambda_1$ and
$\lambda_2>0$ such that
\begin{equation*}\begin{split}  \|P_t-\mu_V\|_{L^\infty(\R^d;\mu_V)\to L^2(\R^d;\mu_V)}&:=\sup_{f\in L^\infty(\R^d;\mu_V)}\|P_tf-\mu_V(f)\|_{L^2(\R^d;\mu_V)}\\
&\le
\exp\big(\lambda_1-\lambda_2t^{1/(1-\sum_{i=1}^d(\varepsilon_i\wedge0))}\big),\quad
t>0.\end{split}\end{equation*} In particular, consider the density function
$$e^{-V(x)}=C\left(\prod_{i=1}^d(1+|x_i|)^{-(1+\alpha)}\right)
\left(\log^{\theta_k}(e+|x_k|)\prod_{1\le i\le d, i\neq
k}\log^{-\varepsilon_i}(e+|x_i|)\right)$$ on $\R^d$ with
$\theta_k>0$, $\varepsilon_i\ge0$ for some $1\le k\le d$ and
any $i\neq k$ and $C$ is the normalizing constant. Then, the
rate function $\alpha$ defined by \eqref{exm3-2} is reduced into
$$\eta(r)=c\Big(1+\log^{\theta_k}(1+r^{-1})\Big),\quad r>0$$ with some constant $c>0$, which is sharp in the sense that
\eqref{thm2-2} does not hold if
$$\lim_{r\to0}\log^{-\theta_k}(1+r^{-1})\eta(r)=0.$$
\end{itemize} \end{corollary}

\begin{proof} (1) Let $$e^{-V_0(x)}
=C_{\varepsilon_1, \ldots, \varepsilon_d,
\alpha}\prod_{i=1}^d(1+|x_i|)^{-(1+(\varepsilon_i\vee\alpha))}.$$
Then, \eqref{thm2-1} holds. On the other hand, for any $s>0$, let
$$A_s=\{x\in\R^d: |x_i|\le s\textrm{ for }1\le i\le d \textrm{ such that }
\varepsilon_i<\alpha\}.$$ Then, there is a constant $c_1>0$ such that for all $s>0$,
$$ \Phi_0(s)= c_1 (1+ s)^{\sum_{i=1}^d(\alpha-\varepsilon_i)^+}.$$ Hence,
for $r>0$ small enough,
\begin{equation*}\begin{split}\eta(r)\le c_2\inf\Big\{(1+ s)^{\sum_{i=1}^d(\alpha-\varepsilon_i)^+}: s>0\textrm{ and  } \sum_{i:0<\varepsilon_i<\alpha}\frac{1}{(1+s)^{\varepsilon_i}} \le c_3r\Big\}.\end{split}\end{equation*}
Setting $s= c_4r^{-1/\varepsilon*}$ for $r>0$ small enough and some constant $c_4>0$ in the right hand side, we can get the
first assertion by Theorem \ref{thm2}. Furthermore, the corresponding
bound of $\|P_t-\mu_V\|_{L^\infty(\R^d;\mu_V)\to L^2(\R^d;\mu_V)}$
follows from \cite[Theorem 4.1.5(2)]{WBook}. Here we mention that, since $\|P_tf-\mu_Vf\|_{L^2(\R^d;\mu_V)}\le 2\|f\|_{L^\infty(\R^d;\mu_V)}$ for all $t>0$, the bound in $t$ is useful only for large $t$. Suppose that the weak
Poincar\'{e} inequality \eqref{thm2-2} holds for a probability
measure
$$e^{-V(x)}=C_{\varepsilon_1, \ldots, \varepsilon_d}
(1+|x_k|)^{-(1+\varepsilon_k)}\prod_{1\le i\le d, i\neq
k}(1+|x_i|)^{-(1+\varepsilon_i)}$$ on $\R^d$ with
$\varepsilon_k\in(0,\alpha)$ and $\varepsilon_i\in[\alpha,\infty)$
for some $1\le k\le d$ and any $i\neq k$. We consider the function
$f(x)=g(x_k)$, where $g\in C_c^\infty(\R)$. Also by \cite[Theorem 2.1(1)]{CW13}, we can apply $f\in C_b^\infty(\R^d)\subset \mathscr{D}(D)$ into the weak Poincar\'{e} inequality
\eqref{thm2-2}, and obtain that for this class of functions the weak Poincar\'{e} inequality
\eqref{thm2-2} is reduced into
\begin{equation}\label{proof-exm3-1}\begin{split}m(g^2)\le& \eta(r)\int_{\R\times \R}
\frac{(g(x+z)-g(x))^2}{|z|^{1+\alpha}}\,dz\,m(dx)\\
&\qquad\qquad \qquad\qquad \qquad+r\|g\|^2_\infty,\quad m(g)=0, g\in
C_c^\infty(\R),
\end{split}\end{equation} where
$m(dx)=C_{\varepsilon_k}(1+|x|)^{-(1+\varepsilon_k)}\,dx$ and
$\varepsilon_k\in (0,\alpha)$. Then, the last assertion is a
consequence of \cite[Corollary 1.2(3)]{WW}.

(2) Let $$e^{-V_0(x)} =C_{\varepsilon_1, \ldots, \varepsilon_d,
\alpha}\prod_{i=1}^d(1+|x_i|)^{-(1+\alpha)}\log^{-(\varepsilon_i\vee0)}(e+|x_i|).$$
Then, \eqref{thm2-1} holds. On the other hand, for any $s>0$, let
$$A_s=\{x\in\R^d: |x_i|\le s\textrm{ for }1\le i\le d \textrm{ such that }
\varepsilon_i<0\}.$$  We can find some constant $c_1>0$ such that for all $s>0$,
$$ \Phi_0(s)=c_1\Big(\log (e+s)\Big)^{-\sum_{i=1}^d(\varepsilon_i\wedge0)},$$ where $-\sum_{i=1}^d(\varepsilon_i\wedge0)>0$. Then, for $r>0$ small enough,
\begin{equation*}\begin{split}\eta(r)\le c_2\inf\Big\{\!\Big(\log (e+s)\Big)^{-\sum_{i=1}^d(\varepsilon_i\wedge0)}: s>0\textrm{ and } \sum_{i:\varepsilon_i<0}\frac{1}{(1+s)^\alpha \log^{\varepsilon_i}(e+s)} \le c_3r\Big\}.\end{split}\end{equation*}
Therefore, we can prove the
first assertion, by using Theorem \ref{thm2} and taking $s=c_4\Big(\frac{1}{r}\log^{-\varepsilon_*}(1+\frac{1}{r})\Big)^{1/\alpha}$ for $r>0$ small enough and some proper constant $c_4>0$ in the right hand side of the inequality above.

Furthermore, the
corresponding bound of $\|P_t-\mu_V\|_{L^\infty(\R^d;\mu_V)\to
L^2(\R^d;\mu_V)}$ follows from \cite[Theorem 4.1.5(1)]{WBook}.
Suppose that the weak Poincar\'{e} inequality \eqref{thm2-2} holds
for a probability measure with the density function as follows
$$e^{-V(x)}=C\left(\prod_{i=1}^d(1+|x_i|)^{-(1+\alpha)}\right)
\left(\log^{\theta_k}(e+|x_k|)\prod_{1\le i\le d, i\neq
k}\log^{-\varepsilon_i}(e+|x_i|)\right),$$ where
$\theta_k>0$ and $\varepsilon_i\ge0$ for some $1\le k\le d$ and
any $i\neq k$. We consider the function $f(x)=g(x_k)$, where $g\in
C_c^\infty(\R)$. Then, the weak Poincar\'{e} inequality
\eqref{thm2-2} is reduced into the inequality \eqref{proof-exm3-1},
where
$$m(dx)=C_{\theta_k,\alpha}(1+|x|)^{-(1+\alpha)}\log^{\theta_k}(e+|x|)\,dx$$
and $\theta_k>0$. Hence, the last assertion follows from
\cite[Corollary 1.3(4)]{WW}.
\end{proof}

\noindent{\bf Acknowledgements.} The author would like to thank two referees for their helpful comments and careful corrections on previous versions. Financial support through National
Natural Science Foundation of China (No.\ 11201073), the JSPS postdoctoral fellowship (26$\cdot$04021), and the Program for Nonlinear Analysis and Its Applications (No. IRTL1206) are gratefully acknowledged.

\end{document}